\documentclass[a4paper,12pt]{amsart}

\usepackage[margin=1in]{geometry}
\usepackage{amsmath,amssymb,amsthm,mathtools}
\usepackage{hyperref}

\newtheorem{theorem}{Theorem}[section]
\newtheorem{lemma}[theorem]{Lemma}
\newtheorem{proposition}[theorem]{Proposition}

\newtheorem{problem}[theorem]{Problem}

\theoremstyle{definition}
\newtheorem{definition}[theorem]{Definition}

\theoremstyle{remark}

\newcommand{\bbR}{\mathbb R}

\newcommand{\bbZ}{\mathbb Z}

\DeclareMathOperator{\RealPart}{Re}
\renewcommand{\Re}{\RealPart}
\DeclareMathOperator{\ImagPart}{Im}
\renewcommand{\Im}{\ImagPart}

\title[Positive-Definiteness and Bernstein Operators]{Preservation of Positive-Definiteness by Bernstein Operators on the Circle}

\author{Matthew Otten}
\address{Department of Physics, University of Wisconsin--Madison, Madison, WI, USA}
\email{mjotten@wisc.edu}

\author{Nam Nguyen}
\address{Department of Physics and Astronomy, George Mason University, Fairfax, VA, 22030 USA}
\address{Quantum Science and Engineering Center, George Mason University, Fairfax, VA, 22030 USA}
\email{nguyen314159265@gmail.com}

\author{Thomas W. Watts}
\address{Centre for Quantum Software and Information,
School of Computer Science, Faculty of Engineering \& Information Technology,
University of Technology Sydney, NSW 2007, Australia}
\email{thomas.watts@student.uts.edu.au}

\date{\today}

\begin{document}

\begin{abstract} 
We prove that, for every $n\ge1$, the degree-$n$ Bernstein operator on $[0,\pi]$ preserves positive-definiteness on the circle $S^1$. Equivalently, if a continuous function on $[0,\pi]$ defines a positive-definite isotropic kernel on $S^1$, then its Bernstein polynomial approximation of any fixed degree does as well. The proof reduces the problem to the nonnegativity of the cosine coefficients of the Bernstein images $ Q_{n,m}=B_n[\cos(mx)]$, which we prove using an explicit coefficient formula and a two-regime positivity argument. We also discuss the higher-dimensional sphere analogue and show that the naive affine Bernstein operator fails to preserve the positive-definite cone already on $S^2$.
\end{abstract} 

\maketitle

\section{Introduction and Main Theorem}
Positive-definite isotropic kernels on the circle are kernels of the form $K(z,w)=f(d_{S^1}(z,w))$, where $d_{S^1}$ denotes geodesic distance. We use the distance variable $x\in[0,\pi]$. In this parametrization, Bochner-Schoenberg theory characterizes positive-definiteness by nonnegative, summable cosine coefficients~\cite{Schoenberg1942PositiveDF,Bingham1973PositiveDF,Katznelson2004Harmonic}.

This positivity condition is also the structural validity condition for covariance and Gram kernels; see, for example, Gneiting~\cite{Gneiting2013Spheres} for the spherical covariance viewpoint. If $K$ is a covariance kernel then every finite sampled matrix $(K(z_p,z_q))_{p,q}$ must be positive semidefinite. Thus, for kernel or covariance models on the circle, an approximation procedure should preserve positive-definiteness, not only pointwise or uniform accuracy. The result below shows that the classical Bernstein approximants have this cone-preservation property on $S^1$ at every fixed degree.

\begin{definition}[Positive-definite isotropic kernels on $S^1$] 
\label{def:pd-circle}
Let $d_{S^1}$ denote geodesic distance on the unit circle, normalized so that $d_{S^1}(z,w)\in[0,\pi]$. A continuous function $f: [0,\pi] \to\mathbb R$ is called positive-definite on $S^1$ if, for every finite set of points $z_1,\ldots,z_N \in S^1$, the matrix 
\[ \bigl(f(d_{S^1}(z_p,z_q))\bigr)_{p,q=1}^N \] 
is positive semidefinite. Equivalently, by the Bochner-Schoenberg characterization, $f$ has a uniformly convergent cosine expansion 
\[ 
f(x) = 
\frac{b_0}{2} + \sum_{m=1}^{\infty} b_m\cos(mx), \qquad b_m\ge0 \ (m \geq 0), \qquad \sum_{m=0}^{\infty} b_m<\infty. 
\] 
Here 
\[ 
b_m=\frac{2}{\pi}\int_0^\pi f(x)\cos(mx) dx,\qquad m \ge 0. 
\] 
Throughout the paper, positive-definite is used in the positive-semidefinite kernel sense, not in the strictly positive-definite sense. 
\end{definition}

For an integer $n\ge 1$, define the Bernstein operator
\begin{equation}
\label{eq:Bndef}
B_n[f](x)=\sum_{j=0}^{n}\binom{n}{j}
\Bigl(\frac{x}{\pi}\Bigr)^j\Bigl(1-\frac{x}{\pi}\Bigr)^{n-j}
f\Bigl(\frac{j\pi}{n}\Bigr), \qquad 0 \le x\le \pi.
\end{equation}
The operator \eqref{eq:Bndef} is the classical Bernstein operator \cite{Lorentz1953Bernstein}. Since $B_n[f]$ is a convex combination of the sampled values $\{f(j\pi/n)\}_{j=0}^n$, it is natural to ask whether this operator also preserves the Bochner--Schoenberg positivity encoded in Definition~\ref{def:pd-circle}.

\begin{theorem}[Main theorem]\label{thm:main}
For every integer $n\ge 1$, $B_n$ preserves positive-definiteness on $S^1$.
\end{theorem}

This result should be distinguished from known coefficient-preserving transforms and dimension-walk operators for spherical positive-definite functions \cite{BeatsonZuCastell2016,BeatsonZuCastell2017,Xu2017Jacobi,MassaPeronPorcu2017}. Those works are tailored to Gegenbauer/Jacobi structures on spheres. The point here is more specific: in the one-dimensional spherical case $S^1$, the classical Bernstein operator itself preserves the nonnegative cosine-coefficient cone.

We derive the required coefficient formula in the next section and prove Theorem~\ref{thm:main} in Section~\ref{sec:full-proof-circle}. Theorem~\ref{thm:main} provides a concrete positivity-preserving approximation mechanism for kernel or covariance-type functions.

\section{Reduction to Explicit Coefficient Inequalities}
We first reduce Theorem~\ref{thm:main} to coefficient inequalities for the
Bernstein images of the cosine basis.

\subsection{Basis reduction}
For later use we record the basic boundedness of the Bernstein operator: for every bounded function $h$ on $[0,\pi]$,
\[
\|B_n[h]\|_\infty\le \|h\|_\infty,
\]
because $B_n[h]$ is a convex combination of sampled values of $h$. Hence uniformly convergent cosine series remain uniformly convergent after applying $B_n$, and termwise integration against $\cos(kx)$ is justified.

Let $f(x)=\frac{b_0}{2} + \sum_{m\ge1} b_m\cos(mx)$ with $b_m\ge0$ for $m \geq 0$ and $\sum_{m = 0}^\infty b_m<\infty$. Uniform convergence and linearity give
\[
  B_n[f](x)=\frac{b_0}{2}+\sum_{m\ge1}b_m\,B_n[\cos(mx)](x).
\]
Hence it is enough to study $B_n[\cos(mx)]$ for each $m$. Thus the problem becomes a cone-preservation statement for the cosine basis: if each basis image has nonnegative cosine coefficients, then any nonnegative summable combination of them does as well.

Set $t=x/\pi\in[0,1]$, define
\[
Q_{n,m}(t):=B_n[\cos(mx)](\pi t),\qquad \zeta:=e^{im\pi/n}.
\]
Using $\cos(m\pi j/n)=\Re(\zeta^j)$ and the binomial theorem,
\[
\begin{aligned}
Q_{n,m}(t)
&=\sum_{j=0}^{n}\binom{n}{j}t^j(1-t)^{n-j}
\cos\!\left(\frac{m\pi j}{n}\right) \\
&=\Re\sum_{j=0}^{n}\binom{n}{j}t^j(1-t)^{n-j}\zeta^j \\
&=\Re[(1-t+t\zeta)^n].
\end{aligned}
\]
Thus
\begin{equation}
\label{eq:Qclosed}
Q_{n,m}(t)=\Re[(1-t+t\zeta)^n].
\end{equation}

For $g$ on $[0,\pi]$, define cosine coefficients
\[
\widehat g_k:=\frac{2}{\pi}\int_0^\pi g(x)\cos(kx)\,dx,\qquad k\ge0.
\]
With this normalization, the cosine expansion is written as
\[
g(x) \sim \frac{\widehat g_0}{2} + \sum_{k\ge1}\widehat g_k\cos(kx).
\]
Thus, for the functions considered below, nonnegativity of all $\widehat g_k$ is exactly nonnegativity of the cosine-series coefficients.

For functions on $[0,1]$, we use the corresponding rescaled cosine coefficient notation 
\[ 
\widehat P_k := 2\int_0^1 P(t)\cos(k\pi t) dt,  \qquad k\ge 0. 
\] 
Equivalently, $\widehat P_k$ is the $k$-th cosine coefficient on $[0,\pi]$ of the function $x\mapsto P(x/\pi)$. In particular, 
\[
\widehat Q_{n,m}(k) = 2\int_0^1 Q_{n,m}(t)\cos(k\pi t) dt, \qquad k\ge 0. 
\]
Note that the rescaling from $x\in[0,\pi]$ to $t\in[0,1]$ is only notational. It removes repeated factors of $\pi$ and does not change the cosine-coefficient positivity question.

\subsection{Endpoint formula}

We first record an endpoint formula for the rescaled cosine coefficients of a polynomial on $[0,1]$. This expresses $\widehat P_k$, and hence later $\widehat Q_{n,m}(k)$, in terms of odd derivatives at the endpoints.

\begin{lemma}\label{lem:endpoint}
If $P$ is a polynomial on $[0,1]$ with $\deg P\le n$, then for every $k\ge1$,
\begin{equation}
\label{eq:endpoint}
\widehat P_k
=2\sum_{j=1}^{\lceil n/2\rceil}(-1)^{j-1}
\frac{(-1)^kP^{(2j-1)}(1)-P^{(2j-1)}(0)}{(k\pi)^{2j}}.
\end{equation}
\end{lemma}

\begin{proof}
Put $\lambda=k\pi$. Since $\sin\lambda=0$ and $\cos\lambda=(-1)^k$, integration by parts gives
\[ 
\int_0^1 P(t)\cos(\lambda t) dt 
= -\frac1\lambda\int_0^1 P'(t)\sin(\lambda t) dt. 
\]
A second integration by parts gives
\[ 
\int_0^1 P'(t) \sin(\lambda t) dt = \frac{P'(0)-(-1)^kP'(1)}{\lambda} + \frac1\lambda\int_0^1 P''(t)\cos(\lambda t) dt. 
\]
Hence 
\[ 
\int_0^1 P(t)\cos(\lambda t) dt = \frac{(-1)^kP'(1)-P'(0)}{\lambda^2} -\frac1{\lambda^2}\int_0^1 P''(t)\cos(\lambda t) dt. 
\]
Repeating this argument yields
\[ 
\int_0^1 P(t)\cos(\lambda t) dt = \sum_{j=1}^{\lceil n/2\rceil}(-1)^{j-1} \frac{(-1)^kP^{(2j-1)}(1)-P^{(2j-1)}(0)}{\lambda^{2j}}. 
\]
The remaining integral vanishes after sufficiently many integrations because $P$ is a polynomial; in the final even-degree case it is a constant multiple of $\int_0^1 \cos(\lambda t) dt = 0$. Multiplying by $2$ gives \eqref{eq:endpoint}.
\end{proof}

We next use the symmetry of $Q_{n,m}$ to relate the derivatives at $t=1$ to those at $t=0$. This will allow us to rewrite \eqref{eq:endpoint} entirely in terms of derivatives at one endpoint.

\begin{lemma}[Reflection symmetry]\label{lem:reflect}
For all $n\ge1$, integers $m$, and $r\ge0$,
\[
Q_{n,m}(1-t)=(-1)^mQ_{n,m}(t),
\qquad
Q_{n,m}^{(r)}(1)=(-1)^{m+r}Q_{n,m}^{(r)}(0).
\]
\end{lemma}

\begin{proof}
Using the Bernstein form,
\[ 
Q_{n,m}(t) = \sum_{j=0}^n {n\choose j}t^j(1-t)^{n-j} \cos\Bigl(\frac{m\pi j}{n}\Bigr). 
\]
Therefore
\[ 
Q_{n,m}(1-t) = \sum_{j=0}^n {n\choose j}(1-t)^jt^{n-j} \cos\Bigl(\frac{m\pi j}{n}\Bigr). 
\]
Changing variables $r=n-j$ gives
\[ 
Q_{n,m}(1-t) = \sum_{r=0}^n {n\choose r}t^r(1-t)^{n-r} \cos\Bigl(\frac{m\pi(n-r)}{n}\Bigr). 
\]
Since 
\[ 
\cos\Bigl(m\pi-\frac{m\pi r}{n}\Bigr) = (-1)^m\cos\Bigl(\frac{m\pi r}{n}\Bigr), 
\]
we obtain
\[ 
Q_{n,m}(1-t)=(-1)^mQ_{n,m}(t). 
\]
Differentiating this identity $r$ times gives
\[ 
(-1)^r Q_{n,m}^{(r)}(1-t) = (-1)^mQ_{n,m}^{(r)}(t). 
\]
Setting $t=0$ gives 
\[ 
Q_{n,m}^{(r)}(1)=(-1)^{m+r} Q_{n,m}^{(r)}(0). 
\]
\end{proof}

\subsection{Explicit coefficient formula}

We now combine the endpoint formula with the reflection symmetry of $Q_{n,m}$ to obtain an explicit expression for the cosine coefficient $\widehat Q_{n,m}(k)$. The point is that Lemma \ref{lem:endpoint} expresses $\widehat Q_{n,m}(k)$ in terms of odd derivatives at $t=0$ and $t=1$, while Lemma~\ref{lem:reflect} reduces the derivatives at $t=1$ to derivatives at $t=0$. The closed form \eqref{eq:Qclosed} then lets us evaluate these derivatives explicitly.

\begin{proposition}\label{prop:explicit}
Fix $n\ge1$, $0\le m\le n$, and set $\theta:=m\pi/(2n)$.
For $k\ge1$:
\[
\widehat Q_{n,m}(k)=0\quad\text{if }k+m\text{ is odd}.
\]
If $k+m$ is even, then
\begin{equation}\label{eq:explicit}
\widehat Q_{n,m}(k)
=4n!\sum_{j=1}^{\lceil n/2\rceil}
\frac{(2\sin\theta)^{2j-1}\sin((2j-1)\theta)}
{(n-2j+1)!(k\pi)^{2j}}.
\end{equation}
\end{proposition}

\begin{proof}
Apply Lemma~\ref{lem:endpoint} to $P=Q_{n,m}$. For $k\ge1$,
\[
\widehat Q_{n,m}(k)
=
2 \sum_{j=1}^{\lceil n/2 \rceil} (-1)^{j-1}
\frac{ 
(-1)^k Q_{n,m}^{(2j-1)}(1) - Q_{n,m}^{(2j-1)}(0) 
}{ 
(k\pi)^{2j} 
}.
\]
By Lemma~\ref{lem:reflect},
\[
Q_{n,m}^{(2j-1)}(1)
=
(-1)^{m+2j-1} Q_{n,m}^{(2j-1)}(0)
=
(-1)^{m+1} Q_{n,m}^{(2j-1)}(0),
\]
so
\[
(-1)^k Q_{n,m}^{(2j-1)}(1) - Q_{n,m}^{(2j-1)}(0)
=
\bigl( (-1)^{k+m+1}-1 \bigr) Q_{n,m}^{(2j-1)}(0).
\]
Therefore
\[
\widehat Q_{n,m}(k)
=
2 \sum_{j=1}^{ \lceil n/2\rceil }(-1)^{j-1}
\frac{ 
\bigl((-1)^{k+m+1}-1\bigr) Q_{n,m}^{(2j-1)}(0)
}{
(k\pi)^{2j}
}.
\]

If $k+m$ is odd, then $(-1)^{k+m+1}-1=0$, hence
\[
\widehat Q_{n,m}(k) = 0.
\]

Assume now that $k+m$ is even. Then $(-1)^{k+m+1}-1=-2$, and so
\[
\widehat Q_{n,m}(k)
=
-4 \sum_{j=1}^{\lceil n/2\rceil}(-1)^{j-1}
\frac{ Q_{n,m}^{(2j-1)}(0) }{ (k\pi)^{2j} }.
\]
From \eqref{eq:Qclosed},
\[
Q_{n,m}^{(r)}(0)
=
\frac{n!}{(n-r)!}\Re[(\zeta-1)^r],
\qquad
\zeta = e^{im\pi/n}.
\]
For $r = 2j-1$ this gives
\[
Q_{n,m}^{(2j-1)}(0)
=
\frac{n!}{(n-2j+1)!}\Re[(\zeta-1)^{2j-1}].
\]
Now write
\[
\zeta=e^{2i\theta},
\qquad
\theta:=\frac{m\pi}{2n}.
\]
Then
\[
\zeta-1
=
e^{i\theta}(e^{i\theta}-e^{-i\theta})
=
2i e^{i\theta}\sin\theta,
\]
hence
\[
(\zeta-1)^{2j-1}
=
(2\sin\theta)^{2j-1} i^{2j-1} e^{i(2j-1)\theta}.
\]
Since
\[
i^{2j-1}=(-1)^{j-1} i,
\]
we obtain
\[
\Re[(\zeta-1)^{2j-1}]
=
(2\sin\theta)^{2j-1} (-1)^{j-1} \Re\bigl(i e^{i(2j-1)\theta}\bigr)
=
(-1)^j (2\sin\theta)^{2j-1} \sin((2j-1)\theta).
\]
Therefore
\[
Q_{n,m}^{(2j-1)}(0)
=
\frac{n!}{(n-2j+1)!}
(-1)^j
(2\sin\theta)^{2j-1}\sin((2j-1)\theta).
\]
Substituting into the previous expression yields
\[
\widehat Q_{n,m}(k)
=
4n! \sum_{j=1}^{ \lceil n/2\rceil}
\frac{ (2\sin\theta)^{2j-1}\sin((2j-1)\theta) }
{ (n-2j+1)!(k\pi)^{2j} },
\]
which is exactly \eqref{eq:explicit}.
\end{proof}

The previous proposition treats the nonzero Fourier modes. We now compute the remaining mode $k=0$ directly from the closed form \eqref{eq:Qclosed}.

\begin{proposition}[The zero Fourier mode]
\label{prop:k0}
For every $n\ge1$ and integer $m$, set
\[
\zeta=e^{im\pi/n}.
\]
Then
\[
\widehat Q_{n,m}(0)
=
\begin{cases}
2, & \zeta=1,\\[1ex]
2\Re\!\left(\dfrac{\zeta^{n+1}-1}{(n+1)(\zeta-1)}\right),
& \zeta\ne1.
\end{cases}
\]
In particular, for $0\le m\le n$ this becomes
\[
\widehat Q_{n,m}(0)=
\begin{cases}
2, & m=0,\\[1ex]
\dfrac{2}{n+1}, & m\ge1\ \text{even},\\[1ex]
0, & m\ \text{odd}.
\end{cases}
\]
Hence $\widehat Q_{n,m}(0)\ge0$ for $0\le m\le n$.
\end{proposition}

\begin{proof}
From the closed form
\[
Q_{n,m}(t)= \Re[( 1-t+t\zeta )^n],
\]
we have
\[
\widehat Q_{n,m}(0)
=
2\int_0^1 Q_{n,m}(t) dt
=
2\Re\int_0^1 (1-t+t\zeta)^n dt.
\]
If $\zeta = 1$, then $Q_{n,m}(t) = 1$, and therefore
\[
\widehat Q_{n,m}(0)=2.
\]
If $\zeta \ne 1$, then
\[
\int_0^1 (1-t+t\zeta)^n dt
=
\frac{\zeta^{n+1}-1}{(n+1)(\zeta - 1)}.
\]
Thus
\[
\widehat Q_{n,m}(0)
=
2\Re\!\left(
\frac{\zeta^{n+1}-1}{(n+1)(\zeta -1)}
\right).
\]

It remains to simplify this expression in the reduced range $0\le m\le n$.
For $m=0$, we already have $\zeta=  1$, so
\[
\widehat Q_{n,0}(0) = 2.
\]
Now assume $1\le m\le n$. Since $\zeta^n = e^{ im\pi } = (-1)^m$, we have
\[
\zeta^{n+1} = (-1)^m\zeta.
\]
If $m$ is even, then $\zeta^{n+1} = \zeta$, and hence
\[
\frac{\zeta^{n+1}-1}{(n+1)(\zeta-1)}
=
\frac{\zeta-1}{(n+1)(\zeta-1)}
=
\frac{1}{n+1}.
\]
Therefore
\[
\widehat Q_{n,m}(0) = \frac{2}{n+1}.
\]
If $m$ is odd, then $\zeta^{n+1} = -\zeta$, so
\[
\frac{\zeta^{n+1}-1}{(n+1)(\zeta-1)}
=
-\frac{\zeta+1}{(n+1)(\zeta-1)}.
\]
Since $|\zeta|=1$ and $\zeta\ne1$, the quantity
\[
\frac{\zeta+1}{\zeta-1}
\]
is purely imaginary. Hence its real part is zero, and therefore
\[
\widehat Q_{n,m}(0) =0.
\]
This proves the stated formula and the nonnegativity of the zero Fourier mode.
\end{proof}

Note that for small Bernstein degrees, direct expansion of the basis images 
\[
Q_{n,m}=B_n[\cos(mx)] 
\] 
already reveals the parity cancellation 
\[ 
\widehat Q_{n,m}(k)=0 \qquad\text{when }k+m\text{ is odd}, 
\] 
as well as the dominance of the first positive term for large output frequency $k$. The edge-near calculations, where $m$ is close to $n$, also suggest the Taylor-remainder and alternating-block mechanisms used in the proof. The proof below gives a uniform treatment of all modes, using Abel summation for $k\ge m$ and an integral representation with a decreasing weight for $m>k$.

\section{Full Proof on the Circle}
\label{sec:full-proof-circle}

We now prove the main positivity statement for the cosine coefficients $\widehat Q_{n,m}(k)$. The explicit finite sum \eqref{eq:explicit} is our starting point. The argument splits into two complementary regimes:
\[
k \ge m
\quad \text{and} \quad
m > k.
\]
In the first regime, the finite sum can be handled directly by an Abel-summation argument. In the second regime, the same finite sum is less convenient, and we instead pass to an integral representation whose kernel has a useful monotonicity property.

We first derive an integral formula for $\widehat Q_{n,m}(k)$ in the parity-compatible case. This representation will be used only in the regime $m > k$.
\begin{proposition}[Integral representation for parity-compatible modes]\label{prop:integral-representation}
Fix $n\ge1$ and $1\le m\le n$, and set
\[
\theta:=\frac{m\pi}{2n}\in\Bigl(0,\frac{\pi}{2}\Bigr].
\]
For $k\ge1$, one has $\widehat Q_{n,m}(k)=0$ whenever $k+m$ is odd.
If $k+m$ is even, define
\[
b:=\frac{k\pi}{2},\qquad a:=b\cot\theta,
\]
then
\begin{equation}\label{eq:integral-rep}
\widehat Q_{n,m}(k)=
\frac{4n}{k\pi}\int_0^1 (1-u)^{n-1}\cosh(au)\sin(bu)\,du.
\end{equation}
\end{proposition}

\begin{proof}
The parity-vanishing statement is already part of
Proposition~\ref{prop:explicit}. Assume $k+m$ is even.

From \eqref{eq:explicit}, write the finite sum as an imaginary part; the auxiliary polynomial $R_n$ below packages the odd powers appearing in the explicit coefficient formula:
\begin{equation}
\label{eq:compact-pn}
\widehat Q_{n,m}(k)=\frac{4n!}{k\pi}\,\Im R_n(z),
\qquad
z:=\frac{\zeta-1}{ik\pi}
=\frac{2\sin\theta}{k\pi}e^{i\theta},
\end{equation}
with $\zeta=e^{im\pi/n}$ and
\[
R_n(z):=\sum_{j=1}^{\lceil n/2\rceil}\frac{z^{2j-1}}{(n-2j+1)!}.
\]
Set $w:=1/z=\frac{k\pi}{2\sin\theta}e^{-i\theta}=a-ib$.

If $n=2p$ is even, reindexing gives
\[
R_n(z)=z^n\sum_{s=0}^{p-1}\frac{w^{2s+1}}{(2s+1)!}.
\]
If $n=2p+1$ is odd, reindexing gives
\[
R_n(z)=z^n\sum_{s=0}^{p}\frac{w^{2s}}{(2s)!}.
\]
Hence, with
\[
G_n(w):=
\begin{cases}
\sinh w,& n\ \text{even},\\
\cosh w,& n\ \text{odd},
\end{cases}
\]
in both parity cases we have
\[
R_n(z)=z^nH_n(w),
\]
where $H_n$ is the $(n-1)$st Taylor polynomial of $G_n$ at $0$.
Taylor's theorem with integral remainder gives
\[
G_n(w) = H_n(w)+\frac{w^n}{(n-1)!}\int_0^1(1-u)^{n-1}\sinh(uw)\,du,
\]
because $G_n^{(n)}=\sinh$ in both parity cases.
Multiplying by $z^n$ and using $z^n w^n=1$,
\[
R_n(z) = z^nG_n(w)-\frac{1}{(n-1)!}\int_0^1(1-u)^{n-1}\sinh(uw)\,du.
\]
Therefore
\begin{equation}\label{eq:im-split}
\Im R_n(z)=\Im\!\bigl(z^nG_n(w)\bigr)
-\frac{1}{(n-1)!}\Im\!\int_0^1(1-u)^{n-1}\sinh(uw)\,du.
\end{equation}

We show $\Im(z^nG_n(w))=0$. Indeed, from
\[ 
z=\frac{2\sin\theta}{k\pi}e^{i\theta}, \qquad \theta=\frac{m\pi}{2n}, 
\]
we can write 
\[ 
z^n=\rho e^{in\theta}=\rho e^{im\pi/2}=\rho i^m, \qquad \rho>0. 
\]
Also $w=a-ib$ with $b=k\pi/2$. If $k$ is even, then $b\in\pi\bbZ$, and
\[ \sinh(a-ib)= \sinh a\cos b-i\cosh a\sin b\in\bbR, \]
\[ \cosh(a-ib)= \cosh a\cos b-i\sinh a\sin b\in\bbR. \]
Since $k+m$ is even, $m$ is also even, so $i^m\in\{\pm1\}$. Thus $z^nG_n(w)\in\bbR$.

If $k$ is odd, then $b\in\pi/2+\pi\bbZ$, and 
\[ \sinh(a-ib)=\sinh a\cos b-i\cosh a\sin b\in i\bbR, \] 
\[ \cosh(a-ib)=\cosh a\cos b-i\sinh a\sin b\in i\bbR. \]
Since $k+m$ is even, $m$ is also odd, so $i^m\in\{\pm i\}$. Thus again $z^nG_n(w)\in\bbR$. Therefore \[ \Im(z^nG_n(w))=0. \]

Also
\[
\sinh(uw)=\sinh(ua)\cos(ub)-i\cosh(ua)\sin(ub),
\]
hence
\[
\Im(\sinh(uw))=-\cosh(ua)\sin(ub).
\]
Using this in \eqref{eq:im-split},
\[
\Im R_n(z)=\frac{1}{(n-1)!}\int_0^1(1-u)^{n-1}\cosh(au)\sin(bu)\,du.
\]
Insert into \eqref{eq:compact-pn} and use $n!/(n-1)!=n$.
\end{proof}

We now treat the regime $k \ge m$. Here the explicit finite sum \eqref{eq:explicit} is already well adapted to the problem, and positivity follows from a monotonicity check on the coefficients together with Abel summation.

\begin{lemma}[Regime $k\ge m$]\label{lem:k-ge-m}
Fix $n\ge1$ and integers $m,k$ with $1\le m\le n$, $k\ge m$, and $k+m$ even.
Then $\widehat Q_{n,m}(k)\ge0$.
\end{lemma}

\begin{proof}
From Proposition~\ref{prop:explicit},
\[
\widehat Q_{n,m}(k)=\sum_{j=1}^{J} c_j\sin((2j-1)\theta),
\qquad
J:=\Bigl\lceil\frac n2\Bigr\rceil,
\]
with
\[
c_j:=\frac{4n!}{(n-2j+1)!}\frac{(2\sin\theta)^{2j-1}}{(k\pi)^{2j}}>0,
\qquad
\theta=\frac{m\pi}{2n}.
\]
If $n=1$ or $n=2$, then $J=1$, so no adjacent ratio occurs; the single summand is nonnegative because $\theta\in(0,\pi/2]$. Thus assume $n\ge3$. Its adjacent ratio is
\[
\frac{c_{j+1}}{c_j}
=\frac{(2\sin\theta)^2}{(k\pi)^2}(n-2j+1)(n-2j).
\]
Since $\sin\theta\le\theta$ and $k\ge m$,
\[
2\sin\theta\le 2\theta=\frac{m\pi}{n}\le\frac{k\pi}{n},
\]
thus 
\[ 
\Bigl(\frac{2\sin\theta}{k\pi}\Bigr)^2\le\frac1{n^2}. 
\]
Whenever the ratio $c_{j+1}/c_j$ occurs, we have  $1\le j\le J-1$. Hence $n-2j\ge0$, and both factors are bounded above by their values at $j=1$:
\[ 
(n-2j+1)(n-2j)\le (n-1)(n-2). 
\] 
Hence 
\[ 
\frac{c_{j+1}}{c_j} \le \frac{(n-1)(n-2)}{n^2}<1 
\] 
for $n \ge 3$. So $c_1\ge c_2\ge\cdots\ge c_J\ge0$.

Let
\[
S_j:=\sum_{r=1}^j \sin((2r-1)\theta)
=\frac{\sin^2(j\theta)}{\sin\theta}\ge0.
\]
Although the individual sine factors may change sign, their odd partial sums are nonnegative. This is the point that makes Abel summation effective. Abel summation gives
\[
\sum_{j=1}^{J} c_j\sin((2j-1)\theta)
=c_JS_J+\sum_{j=1}^{J-1}(c_j-c_{j+1})S_j\ge0.
\]
Hence $\widehat Q_{n,m}(k)\ge0$.
\end{proof}

We next turn to the regime $m > k$. In this range, the coefficient sequence in \eqref{eq:explicit} is no longer the right object to control directly, so we use the integral representation from Proposition~\ref{prop:integral-representation}. The point is that the weight in that integral is decreasing, which lets us exploit the oscillation of $\sin(bu)$.

\begin{lemma}[Regime $m>k$]\label{lem:m-gt-k}
Fix $n\ge1$ and integers $m,k$ with $1\le m\le n$, $1\le k<m$, and $k+m$ even.
Then $\widehat Q_{n,m}(k)\ge0$.
\end{lemma}

\begin{proof}
By Proposition~\ref{prop:integral-representation},
\[
\widehat Q_{n,m}(k)=\frac{4n}{k\pi}\int_0^1 w(u)\sin(bu)\,du,
\]
where
\[
w(u):=(1-u)^{n-1}\cosh(au),\qquad
b:=\frac{k\pi}{2},\qquad
a:=b\cot\theta, \qquad \theta=\frac{m\pi}{2n}.
\]
Since $m>k\ge1$ and $m\le n$, necessarily $n\ge2$.

First, $w$ is nonnegative and nonincreasing on $[0,1]$.
Indeed
\[
w'(u)=(1-u)^{n-2}\!\left(-(n-1)\cosh(au)+(1-u)a\sinh(au)\right).
\]
So it is enough to show $(1-u)a\tanh(au)\le n-1$. Here $a \ge 0$, since $\cot\theta \ge 0$ on $(0,\pi/2]$. Thus $0\le\tanh(au)\le1$, and from $(1-u) \le 1$ it suffices that $a\le n-1$. If \(\theta=\pi/2\), then \(a = b\cot\theta =0\), so \(a\le n-1\). Otherwise \(0< \theta < \pi/2\), and \(\tan\theta\ge \theta\), hence \(\cot\theta\le 1/\theta\). Hence,
\[
a=\frac{k\pi}{2}\cot\theta\le \frac{k\pi}{2}\cdot\frac1\theta
=\frac{k\pi}{2}\cdot\frac{2n}{m\pi}=\frac{kn}{m}.
\]
Because $m\ge k+1$ and $k\le n-1$,
\[
\frac{kn}{m}\le\frac{kn}{k+1}=n-\frac{n}{k+1}\le n-1.
\]
Hence $w'(u)\le0$.

Now set $t=bu$, so
\[
\int_0^1 w(u)\sin(bu)\,du
=\frac1b\int_0^b W(t)\sin t dt,\qquad W(t):=w(t/b).
\]
Thus $W$ is nonnegative and decreasing on $[0,b]$.

We now use the elementary fact that integrating $\sin t$ against a nonnegative decreasing weight over full or half periods gives a nonnegative alternating-block sum. If $k = 2\ell$ is even, then $b = \ell\pi$ and
\[
\int_0^{\ell\pi}W(t)\sin t dt
= \sum_{r=0}^{\ell-1}(-1)^rM_r,\quad
M_r:= \int_0^\pi W(r\pi+u) \sin u du.
\]
Since $W$ decreases, $M_0 \ge M_1\ge\cdots\ge M_{\ell-1} \ge 0$. Therefore, if $\ell$ is even, then 
\[ 
\sum_{r=0}^{\ell-1}(-1)^rM_r = (M_0-M_1) + \cdots + (M_{\ell-2}-M_{\ell-1}) \ge 0.
\]
If $\ell$ is odd, then the sum is $M_0 \ge 0$ when $\ell=1$, while for $\ell\ge3$,
\[
\sum_{r=0}^{\ell-1}(-1)^rM_r
= (M_0-M_1) + \cdots + (M_{\ell-3}-M_{\ell-2})+M_{\ell-1} \ge 0.
\]
Thus the alternating sum is nonnegative in all even-$k$ cases.

If $k = 2 \ell +1$ is odd, then $b = \ell\pi + \pi/2$. For $\ell=0$ (i.e. $k=1$),
\[ 
\int_0^{\pi/2}W(t) \sin t dt \ge 0. 
\]
Now assume $\ell \ge 1$. Then 
\[ 
\int_0^{\ell\pi+\pi/2}W(t)\sin t dt = \sum_{r=0}^{\ell-1}(-1)^rM_r + (-1)^\ell C, 
\] 
where
\[ 
M_r =\int_0^\pi W(r\pi+u)\sin u\,du, \qquad C=\int_0^{\pi/2}W(\ell\pi+u) \sin u du. 
\]
Since $W$ is decreasing and nonnegative, 
\[ 
M_0\ge M_1\ge\cdots\ge M_{\ell-1}\ge 0, \qquad 0\le C\le M_{\ell-1}. 
\]

If $\ell$ is even, then 
\[ 
\sum_{r=0}^{\ell-1}(-1)^rM_r = (M_0-M_1)+\cdots+(M_{\ell-2}-M_{\ell-1})\ge0, \] 
and the final term is $+C$. Hence the integral is nonnegative. 

If $\ell$ is odd, then the final term is $-C$. For $\ell=1$, the paired sum is empty and the integral is $M_0-C\ge0$. For $\ell\ge3$,
\[ 
\sum_{r=0}^{\ell-1}(-1)^rM_r = (M_0-M_1) + \cdots + (M_{\ell-3}-M_{\ell-2})+M_{\ell-1} \ge M_{\ell-1}. 
\] 
Since $C \le M_{\ell-1}$, the whole expression is still nonnegative. Thus the integral is nonnegative in all odd-$k$ cases. Therefore 
\[ 
\int_0^1 w(u)\sin(bu) du \ge 0, 
\]
and multiplying by $4n/(k\pi)>0$ yields $\widehat Q_{n,m}(k) \ge 0$.

\end{proof}

\begin{proof}[Proof of Theorem~\ref{thm:main}]
Fix $n\ge1$. By basis reduction, it is enough to prove
\[
\widehat Q_{n,m}(k)\ge0
\quad\text{for all }m\ge0,\ k\ge0.
\]
For each $m\ge0$, let $r$ be the representative of $m$ modulo
$2n$ in $\{0,\ldots,2n-1\}$, and set
\[
m^\sharp :=
\begin{cases}
r, & 0\le r\le n,\\
2n-r, & n<r<2n.
\end{cases}
\]
This is the aliasing induced by the Bernstein nodes: the sampled values of $\cos(mx)$ at $x=j\pi/n$ depend on $m$ only modulo $2n$, up to reflection. Then for every $j=0,\dots,n$,
\[
\cos\!\Bigl(\frac{m\pi j}{n}\Bigr)
=\cos\!\Bigl(\frac{m^\sharp\pi j}{n}\Bigr),
\]
so $Q_{n,m}=Q_{n,m^\sharp}$ and hence
$\widehat Q_{n,m}(k)=\widehat Q_{n,m^\sharp}(k)$ for all $k$.
Therefore it is enough to treat $0\le m\le n$.

For $k=0$, Proposition~\ref{prop:k0} gives nonnegativity.
For $k\ge1$:
if $m=0$, then $Q_{n,0}\equiv1$, so $\widehat Q_{n,0}(k)=0$.
Assume $1\le m\le n$.
If $k+m$ is odd, Proposition~\ref{prop:explicit} gives
$\widehat Q_{n,m}(k)=0$.
If $k+m$ is even and $k\ge m$, use Lemma~\ref{lem:k-ge-m}.
If $k+m$ is even and $m>k$, use Lemma~\ref{lem:m-gt-k}.
Hence $\widehat Q_{n,m}(k)\ge0$ for all $m,k$.

Now let
\[
f(x)=\frac{b_0}{2}+\sum_{m\ge1}b_m\cos(mx),\qquad b_m\ge0,\qquad \sum_{m=0}^\infty  b_m<\infty.
\]
By linearity and uniform convergence,
\[
B_n[f](x) = \frac{b_0}{2}+\sum_{m\ge1}b_m  B_n[\cos(mx)](x).
\]
Uniform convergence also justifies integrating term-by-term against $\cos(kx)$. Therefore the $k$-th cosine coefficient of $B_n[f]$ is
\[ 
\widehat{B_n[f]}_k = \sum_{m\ge1} b_m \widehat Q_{n,m}(k) 
\] 
for $k\ge1$, while the zeroth coefficient is 
\[ 
\widehat{B_n[f]}_0 = b_0+\sum_{m\ge1}b_m \widehat Q_{n,m}(0). 
\]
Since $b_m\ge0$ and $\widehat Q_{n,m}(k)\ge0$ for all $m$ and $k$, every term on the right-hand sides is nonnegative. Hence
\[ \widehat{B_n[f]}_k\ge0 \qquad\text{for all }k\ge0. \]

Moreover, $B_n[f]$ is a polynomial on $[0,\pi]$, and Lemma~\ref{lem:endpoint}, applied after rescaling to $[0,1]$, shows that its nonzero cosine coefficients are $O(k^{-2})$. Hence these coefficients are summable, and the cosine series of $B_n[f]$ converges uniformly to $B_n[f]$. By Bochner's theorem on the circle (equivalently, the characterization of positive-definite functions on $S^1$ by nonnegative cosine coefficients), it follows that $B_n[f]$ is positive-definite on $S^1$.
\end{proof}

\section{Spheres and Power-Bernstein Operators}

The finite-dimensional sphere analogue remains open. Let $d \ge 2$ and $\lambda = (d-1)/2 $. A continuous function $\varphi:[-1,1] \to \mathbb R$ defines a zonal positive-definite kernel on $S^d$ if
\[
K(x,y)=\varphi(x\cdot y)
\]
is positive semidefinite for every finite set of points on $S^d$. By Schoenberg's theorem \cite{Schoenberg1942PositiveDF,Gneiting2013Spheres}, this is equivalent to a Gegenbauer expansion
\[
\varphi(t)=\sum_{\ell=0}^{\infty} a_\ell
\frac{C_\ell^{(\lambda)}(t)}{C_\ell^{(\lambda)}(1)},
\qquad a_\ell\ge0,
\qquad \sum_{\ell=0}^{\infty}a_\ell<\infty.
\]
Here $C_\ell^{(\lambda)}$ denotes the Gegenbauer polynomial of degree $\ell$; for $d=2$, $\lambda=1/2$ and $C_\ell^{(1/2)}=P_\ell$, the Legendre polynomial.

\begin{problem}[Finite-dimensional sphere analogue]
\label{conj:spheres}
Let $d\ge2$. Is there a natural Bernstein-type approximation operator $T_n^{(d)}$ on zonal functions $\varphi:[-1,1]\to\mathbb R$ such that, whenever $\varphi(x\cdot y)$ is positive-definite on $S^d$, the kernel
\[
(x,y)\mapsto (T_n^{(d)}\varphi)(x\cdot y)
\]
is also positive-definite on $S^d$? Equivalently, can one construct a Bernstein-type operator that preserves nonnegativity of the Gegenbauer coefficients?
\end{problem}

Note that the term ``Bernstein-type'' is an important distinction here. The most direct affine Bernstein analogue in the dot-product variable already fails on $S^2$. To avoid confusing this interval operator with the circle operator $B_n$ from \eqref{eq:Bndef}, write $B_4^{[-1,1]}$ for the ordinary affine Bernstein operator on $[-1,1]$,
\[
B_4^{[-1,1]}[f](t) =\sum_{j=0}^{4} f\!\left(-1+\frac{j}{2}\right)
\binom{4}{j}\left(\frac{1+t}{2}\right)^j
\left(\frac{1-t}{2}\right)^{4-j}.
\]
A direct calculation gives
\[
B_4^{[-1,1]}[P_6](t) = \frac{347}{2048}+\frac{63}{64}t^2-\frac{315}{2048}t^4,
\]
and the Legendre expansion is
\[
B_4^{[-1,1]}[P_6]
=\frac{239}{512}P_0+\frac{291}{512}P_2-\frac{9}{256}P_4.
\]
The input $P_6$ is positive-definite on $S^2$, since its Legendre expansion has the single nonzero Schoenberg coefficient $a_6=1$. By uniqueness of the Legendre expansion, the negative $P_4$ coefficient shows that this naive affine Bernstein operator does not preserve the positive-definite cone on $S^2$. Thus any successful higher-dimensional analogue must be adapted to the Gegenbauer coefficient structure rather than obtained by a direct affine transplant of the interval operator.

The finite-dimensional problem is difficult because the positive-definite cone is described by Gegenbauer coefficients, and a Bernstein-type operator would have to preserve that coefficient cone. In the limiting case $S^\infty$, however, the coefficient cone is much simpler. Schoenberg's characterization says that a continuous zonal function $g:[-1,1]\to\mathbb R$ is positive-definite on $S^\infty$ if and only if 
\[ 
g(t)=\sum_{r=0}^{\infty} \beta_r t^r, \qquad \beta_r\ge0, \qquad \sum_{r=0}^{\infty}\beta_r<\infty. 
\] 
Thus, on $S^\infty$, it is enough to preserve nonnegativity of ordinary monomial coefficients. This gives a simple positive result for a related power-Bernstein family. In particular, for a positive integer $\alpha$, define
\[
B_{\alpha,n}[g](t):= \sum_{j=0}^{n}
g \left(\frac{j}{n}\right)
\binom{n}{j}t^{\alpha j}(1-t^\alpha)^{n-j}, \qquad -1\le t\le1.
\]

\begin{proposition}[Power-Bernstein preservation on $S^\infty$]
\label{prop:power-bernstein-sinfty}
Let $\alpha \in \{1,2,3,\ldots\}$. If
\[
g(t) = \sum_{r=0}^{\infty}\beta_r t^r, 
\qquad \beta_r\ge0,
\qquad \sum_{r=0}^{\infty}\beta_r < \infty,
\]
then $B_{\alpha,n}[g]$ has an expansion
\[
B_{\alpha,n}[g](t) = \sum_{\ell=0}^{n} \Gamma_\ell t^{\alpha\ell},
\qquad \Gamma_\ell \ge 0.
\]
Consequently, $B_{\alpha,n}[g]$ is positive-definite on $S^\infty$.
\end{proposition}

\begin{proof}
By linearity and uniform convergence of the power series for $g$,
\[
B_{\alpha,n}[g](t)
=
\sum_{r\ge0} \beta_r \sum_{j=0}^{n}\binom{n}{j} t^{\alpha j}(1-t^\alpha)^{n-j}(j/n)^r.
\]
Set $y = t^\alpha$. Then
\[
B_{\alpha,n}[g](t)
=
\sum_{r \ge 0} \beta_r \sum_{j=0}^{n}\binom{n}{j} y^j(1-y)^{n-j}(j/n)^r.
\]
Expanding $(1-y)^{n-j}$ and collecting powers of $y$ gives
\[
B_{\alpha,n}[g](t) 
= \sum_{\ell=0}^{n}\Gamma_\ell y^\ell
= \sum_{\ell=0}^{n}\Gamma_\ell t^{\alpha\ell},
\]
where
\[
\Gamma_\ell
=
\sum_{r\ge 0} \beta_r \sum_{j=0}^{\ell} \binom{n}{j}\binom{n-j}{\ell-j} (-1)^{\ell-j}(j/n)^r.
\]
Using
\[
\binom{n}{j}\binom{n-j}{\ell-j}
=
\binom{n}{\ell}\binom{\ell}{j},
\]
we obtain
\[
\Gamma_\ell
=
\sum_{r\ge0}\beta_r\binom{n}{\ell}\frac{1}{n^r}
\sum_{j=0}^{\ell}(-1)^{\ell-j}\binom{\ell}{j}j^r.
\]
The inner sum is the standard finite-difference representation of the Stirling numbers:
\[
\sum_{j=0}^{\ell}(-1)^{\ell-j}\binom{\ell}{j}j^r = \ell! S(r,\ell),
\]
where $S(r,\ell)$ is a Stirling number of the second kind, with the conventions $S(0,0)=1$ and $S(r,0)=0$ for $r>0$. These conventions account for the $\ell=0$ terms. Hence
\[
\Gamma_\ell
=
\sum_{r\ge0}\beta_r\binom{n}{\ell}\frac{\ell!\,S(r,\ell)}{n^r}
\ge0.
\]
Since $\alpha$ is a positive integer, each exponent $\alpha\ell$ is a
nonnegative integer. Therefore $B_{\alpha,n}[g]$ has a nonnegative integer-power
expansion, so by Schoenberg's characterization it is positive-definite on
$S^\infty$.
\end{proof}

\section{Conclusion and Remarks}
Theorem~\ref{thm:main} shows that the classical Bernstein operator, although defined by point samples rather than by Fourier multipliers, preserves the Bochner-Schoenberg positive-definite cone on $S^1$. This matters because positive-definite functions on the circle arise as isotropic covariance and kernel functions, where approximation should not destroy positive semidefiniteness of finite sampled matrices. Equivalently, Bernstein approximation gives a way to approximate an admissible circle covariance without leaving the covariance cone at any finite degree. The proof is coefficient-theoretic: the Bernstein images of the cosine basis have nonnegative cosine coefficients, and the two-regime argument extracts the monotonicity and alternating-block structure needed for this nonnegativity.

The higher-dimensional discussion shows both the promise and the limitation of this circle result. The naive affine Bernstein analogue fails on $S^2$, while the limiting space $S^\infty$ admits a simple integer power-Bernstein construction. A genuine finite-dimensional analogue should therefore be formulated as a Gegenbauer-coefficient-preservation problem rather than as a direct interval-to-sphere transplant.

\section*{Acknowledgments}
We thank Tianshi Lu and Chunsheng Ma of Wichita State University for helpful early discussions of this work.

The authors used ChatGPT and Codex to assist with exploratory symbolic and computational checks in small cases, which helped identify patterns later generalized in the proof. The mathematical statements, proofs, references, and conclusions were independently reviewed and verified by the authors. The authors are responsible for all content of the manuscript, including any errors.

\bibliographystyle{unsrt}
\bibliography{refs}

\end{document}